\documentclass{amsart}%
\usepackage{amssymb}
\usepackage{amsmath, color}
\usepackage{amsfonts}
\usepackage[backref]{hyperref}
\usepackage[all,cmtip]{xy}
\usepackage[margin=1.15in]{geometry}
\usepackage{graphicx}%
\usepackage{amsmath}%
\newtheorem{theorem}{Theorem}
\theoremstyle{plain}

\newtheorem{definition}{Definition}
\newtheorem{lemma}{Lemma}

\newtheorem{proposition}{Proposition}
\theoremstyle{remark}
\newtheorem{example}{Example}

\newtheorem{remark}{Remark}

\numberwithin{equation}{section}
\begin{document}
\title[Automorphisms and class groups]{Automorphisms of OT\ manifolds and ray class numbers}
\author{Oliver Braunling}
\address{FRIAS, Albert Ludwig University of Freiburg, Albertstra\ss e 19, 79104
Freiburg im Breisgau, Germany}
\email{oliver.braeunling@math.uni-freiburg.de}
\author{Victor Vuletescu}
\address{Victor Vuletescu, University of Bucharest, Faculty of Mathematics, 14
Academiei str., 70109 Bucharest, Romania}
\email{vuli@fmi.unibuc.ro}
\thanks{O.B. was supported by the GK1821 \textquotedblleft Cohomological Methods in
Geometry\textquotedblright\ and a FRIAS\ Junior Fellowship.}
\thanks{V.V. was supported by a grant of Ministry of Research and Innovation, CNCS -
UEFISCDI, project number PN-III-P4-ID-PCE-2016-0065, within PNCDI III.}

\begin{abstract}
We compute the automorphism group of OT manifolds of simple type. We show that
the graded pieces under a natural filtration are related to a certain ray
class group of the underlying number field. This does not solve the open
question whether the geometry of the OT manifold sees the class number
directly, but brings us a lot closer to a possible solution.

\end{abstract}
\maketitle

Let $K$ be a number field with $s\geq1$ real places and $t\geq1$ complex
places. For suitable choices of a subgroup $U\subseteq\mathcal{O}_{K}%
^{\times,+}$ of the totally positive units, there is a properly discontinuous
action of $\mathcal{O}_{K}\rtimes U$ on $\mathbb{H}^{s}\times\mathbb{C}^{t}$,
essentially based on embedding $K$ via its infinite places and letting the
group act by addition and multiplication. The key point is that
\begin{equation}
X(K,U):=(\mathbb{H}^{s}\times\mathbb{C}^{t})\left.  /\right.  (\mathcal{O}%
_{K}\rtimes U)\label{lcc6}%
\end{equation}
becomes a compact complex manifold, a so-called\ \emph{Oeljeklaus--Toma
manifold} (or \textquotedblleft OT\ manifold\textquotedblright)
\cite{MR2141693}, \cite{MR3195237}. These manifolds are, in a way,
higher-dimensional analogues of the type $S^{0}$ Inoue surfaces, one of the
better understood types among the Class VII$_{0}$ surfaces in Kodaira's classification.

If $t=1$, then knowing $X:=X(K,U)$, even just its fundamental group, suffices
to reconstruct the number field $K$ uniquely. This can fail for $t>1 $.
However, whenever $K$ is fully determined by $X$, it is natural to ask whether
one can read off the arithmetic invariants of $K$ directly from the geometry
of $X$.

So far, even for $t=1$, it is \textit{not known} how to read off the class
group or even just the class number of $K$ from $X$. At the same time, several
other invariants are readily accessible, e.g., if $X$ is an OT\ manifold of
simple type:%
\[%
\begin{tabular}
[c]{l|l}%
Geometry of $X$ & Arithmetic of $K$\\\hline
dimension & $s+t$\\
Betti number $b_{1}$ & $s$\\
Betti number $b_{2}$ & $\frac{1}{2}s(s-1)$\\
LCK rank & $s$ (not CM) or $\frac{s}{2}$ (if $K$ is CM)\\
$h^{1,0}$ & $0$\\
$h^{0,1}$ & $\geq s$\\
normalized volume & $\sim_{\mathbb{Q}^{\times}}\sqrt{\left\vert
\text{discriminant}\right\vert }\cdot$regulator\\
admits LCK metric & if and only if\\
& $\quad\left\vert \sigma_{i}(u)\right\vert =\left\vert \sigma_{j}%
(u)\right\vert $ for all $\sigma_{i},\sigma_{j},u$\\
$H_{1}(X,\mathbb{Z})$ & $U\times(\mathcal{O}_{K}/J)$ for certain ideal $J$\\
\quad? & field automorphisms $\operatorname{Aut}(K/\mathbb{Q})$\\
\quad? & class group
\end{tabular}
\
\]
(above, $\sigma_{i},\sigma_{j}$ refers to genuine complex places, i.e. those
complex embeddings whose image does not lie in the reals, and $u$ refers to
any $u\in U$).

We refer to \cite{MR2141693} or \cite{MR3195237} for unexplained terminology.

We will not be able to solve this open problem, but we find invariants which
are of the same arithmetic nature as class groups: ray class groups. This data
turns out to be encoded in the holomorphic automorphism group of $X$.

\begin{theorem}
\label{thm_Aut copy(1)}Suppose $X:=X(K,U)$ is an OT manifold of simple type.
Then the biholomorphism group $\operatorname{Aut}(X)$ is canonically
isomorphic to%
\[
\left(  \left.  \left(  \frac{(\mathcal{O}_{K}:J(U))}{\mathcal{O}_{K}}\right)
\rtimes\left(  \frac{\mathcal{O}_{K}^{\times,+}}{U}\right)  \right.  \right)
\rtimes A_{U}\text{.}%
\]
(We define $A_{U}$ and $J(U)$ in the main body of the paper.) More concretely,
it has a canonical ascending three step filtration $F_{0}\subseteq
F_{1}\subseteq F_{2}$ whose graded pieces are%
\begin{align*}
\operatorname{gr}_{0}^{F}\operatorname{Aut}(X) &  \simeq\mathcal{O}%
_{K}/J(U)\text{;}\\
\operatorname{gr}_{1}^{F}\operatorname{Aut}(X) &  \cong\mathcal{O}_{K}%
^{\times,+}/U\text{;}\\
\operatorname{gr}_{2}^{F}\operatorname{Aut}(X) &  \cong A_{U}\text{.}%
\end{align*}
For $\operatorname{gr}_{0}^{F}\operatorname{Aut}(X)$ the isomorphism is
non-canonical, while for $\operatorname{gr}_{i}^{F}\operatorname{Aut}(X)$ with
$i=1,2$ it is.
\end{theorem}

See Theorem \ref{thm_Aut}. These graded pieces may look innocuous, but let us
point out that they are related to class-number like invariants of $K$.

\begin{theorem}
\label{thm_m copy copy(1)}Let $K$ be a number field with $s\geq1$ real places
and precisely one complex place. Moreover, suppose $\mathfrak{m}$ is a modulus
such that its finite part $\mathfrak{m}_{0}$ satisfies $J(U_{\mathfrak{m}%
,1})=\mathfrak{m}_{0}$. Then the ray unit group $U_{\mathfrak{m},1}$ is an
admissible subgroup and let $X:=X(K,U_{\mathfrak{m},1})$ be the corresponding
OT manifold. The graded Euler characteristic%
\[
\chi^{F}(\operatorname{Aut}(X))=\prod_{i}\left\vert \operatorname{gr}_{i}%
^{F}\operatorname{Aut}(X)\right\vert ^{(-1)^{i}}%
\]
satisfies%
\[
\frac{h_{\mathfrak{m}}}{h}\leq\frac{\chi^{F}(\operatorname{Aut}(X))}%
{\left\vert A_{U_{\mathfrak{m},1}}\right\vert }\text{,}%
\]
where $h_{\mathfrak{m}}$ denotes the ray class number of $\mathfrak{m}$ and
$h$ is the ordinary class number.
\end{theorem}

See Theorem \ref{thm_m copy}. The statement of this result uses some
definitions and jargon from number theory, which we shall review and summarize
to the extent needed in Section \ref{sect_RCFT}.

Although not connected to the principal question of this paper, we also obtain
the following result:

\begin{theorem}
Let $K$ be a number field with $s=t=1$ embeddings and $u$ a (totally) positive
fundamental unit, i.e. $\mathcal{O}_{K}^{\times,+}\simeq\mathbb{Z}\left\langle
u\right\rangle $. Then all groups%
\[
U_{n}:=\mathbb{Z}\left\langle u^{n}\right\rangle
\]
are admissible, and for $X_{n}:=X(K,U_{n})$ we have%
\[
\lim_{n\longrightarrow\infty}\frac{\log\left\vert H_{1}(X_{n},\mathbb{Z}%
)_{\operatorname*{tor}}\right\vert }{n}=\log M(f)\text{,}%
\]
where $f$ is the minimal polynomial of the unit $u$, and $M(f)$ denotes the
Mahler measure of $f$. In particular, $\left\vert H_{1}(X_{n},\mathbb{Z}%
)_{\operatorname*{tor}}\right\vert $ always grows asymptotically exponentially
as $n\rightarrow+\infty$.
\end{theorem}

See Theorem \ref{thm_TorsionMahler}. This essentially means that all Inoue
surfaces of type $S^{0}$ sit inside a tree of covering spaces, which are
itself Inoue surfaces of type $S^{0}$, and as we go further into the branches
of the tree, the torsion in $H_{1}$ always grows exponentially, for example%
\[
\cdots\longrightarrow X_{2^{3}}\longrightarrow X_{2^{2}}\longrightarrow
X_{2}\longrightarrow X_{1}\text{,}%
\]
or similarly for all other primes, or other chains of numbers totally ordered
by divisibility. This is not so relevant for our question around class
numbers, but it explains why Inoue surfaces tend to have so much torsion
homology.\medskip

Convention: The word \emph{ring} always means a unital commutative and
associative ring.

\section{Biholomorphisms}

Let $K$ be a number field with $s\geq1$ real places and $t\geq1$ complex
places. Let $U\subseteq\mathcal{O}_{K}^{\times,+}$ be an \emph{admissible}
subgroup, i.e. a rank $s$ free abelian subgroup (see \cite[\S 1]{MR2141693}).
We call $U$ \emph{of simple type} if $K=\mathbb{Q}(u\mid u\in U)$, or
equivalently if there is no proper subfield of $K$ which already contains $U$.

\begin{definition}
\label{Old_def_IdealJ}Suppose $U\subseteq\mathcal{O}_{K}^{\times}$ is an
arbitrary subgroup. Then define%
\[
J(U):=\{\text{\emph{ideal of }}\mathcal{O}_{K}\text{\emph{ generated by }%
}u-1\text{\emph{\ for all} }u\in U\}\text{.}%
\]

\end{definition}

\begin{definition}
We define the fractional ideal%
\begin{equation}
(\mathcal{O}_{K}:J(U)):=\left\{  \beta\in K\mid\forall u\in U:(1-u)\beta
\in\mathcal{O}_{K}\right\}  \text{.}\label{tsyma1}%
\end{equation}
This is also sometimes denoted by $J(U)^{-1}$.
\end{definition}

(From a commutative algebra standpoint, this fractional ideal is the inverse
of $J(U)$ in the sense of invertible $\mathcal{O}_{K}$-modules.)

\begin{definition}
\label{def_AKU}We define $A_{U}:=\{g\in\operatorname*{Aut}(K/\mathbb{Q})\mid
gU=U\}$.
\end{definition}

By $gU=U$ we mean that $g$ sends elements in $U$ to elements in $U$, and not
that it would element-wise fix $U$, and $\operatorname*{Aut}(K/\mathbb{Q})$
denotes field automorphisms of $K$.

When we write $\mathcal{O}_{K}\rtimes U$, we mean the semi-direct product of
the abelian groups $(\mathcal{O}_{K},+)$ and $U$, where $U\subseteq
\mathcal{O}_{K}^{\times}$ acts on $(\mathcal{O}_{K},+)$ by multiplication with
respect to the ring structure of $\mathcal{O}_{K}$. Each element of
$\mathcal{O}_{K}\rtimes U$ can uniquely be written as a pair $(a,b)$ with
$a\in\mathcal{O}_{K}$ and $b\in U$.

For a group $G$, let us write $G_{\operatorname*{ab}}$ for its abelianization,
and if $G$ is abelian, write $G_{\operatorname*{tor}}$ for the subgroup of
torsion elements, and $G_{\operatorname*{fr}}:=G/G_{\operatorname*{tor}}$ for
the torsion-free quotient.

\begin{proposition}
\label{Prop_KappaAgreesWithOK}Let $K$ be a number field and $U\subseteq
\mathcal{O}_{K}^{\times}$ an admissible subgroup. Then for $\pi:=\mathcal{O}%
_{K}\rtimes U$, the kernel $\varkappa$ in the short exact sequence%
\begin{equation}
1\longrightarrow\varkappa\longrightarrow\pi\longrightarrow\pi
_{\operatorname*{ab},\operatorname*{fr}}\longrightarrow1\label{tsy1}%
\end{equation}
is precisely the subgroup $\mathcal{O}_{K}$ appearing in the definition of
$\pi$ as a semi-direct product. The commutator subgroup is%
\[
\lbrack\pi,\pi]=\{\text{pairs }(u,1)\in\pi\mid u\in J(U)\}\text{.}%
\]
Moreover, we have a canonical short exact sequence%
\begin{equation}
0\longrightarrow\mathcal{O}_{K}/J(U)\longrightarrow H_{1}(X(K;U),\mathbb{Z}%
)\longrightarrow U\longrightarrow0\text{,}\label{tsy2}%
\end{equation}
where $\mathcal{O}_{K}/J(U)$ is precisely the torsion subgroup of the middle
term. In particular, this group needs at most $s+2t$ generators.
\end{proposition}

\begin{proof}
A proof is given in \cite[Prop. 6]{otarith}, based on \cite[Thm.
4.2]{MR2875828}. Loc. cit. requires $U$ to be a torsion-free subgroup, but
since $K$ has at least one real place and $\sigma:K\hookrightarrow
\mathbb{R}^{\times}$ is injective, either $\mathcal{O}_{K,\operatorname*{tor}%
}^{\times}$ is trivial or we have $\mathcal{O}_{K,\operatorname*{tor}}%
^{\times}=\left\langle -1\right\rangle $. Either way, the subgroup of totally
positive units $\mathcal{O}_{K}^{\times,+}$ is necessarily torsion-free.
\end{proof}

Let $K$ be a number field and $U\subseteq\mathcal{O}_{K}^{\times}$ an
admissible subgroup. Write $X:=X(K,U)$ as in Equation \ref{lcc6} to denote the
corresponding Oeljeklaus--Toma manifold. Let us write%
\[
\underline{\sigma}:\mathcal{O}_{K}\hookrightarrow\mathbb{C}^{s}\times
\mathbb{C}^{t}\qquad\text{for the map}\qquad\alpha\mapsto(\sigma_{1}%
(\alpha),\ldots,\sigma_{s+t}(\alpha))\text{,}%
\]
where $\sigma_{1},\ldots,\sigma_{s}:K\rightarrow\mathbb{R}$ denote the real
embeddings, and $\sigma_{s+1},\ldots,\sigma_{s+t}:K\rightarrow\mathbb{C}$ one
representative for each complex conjugate pair of the genuinely complex embeddings.

\begin{lemma}
\label{Lemma_TranslationAction}There are three constructions which naturally
give biholomorphisms of $X$.

\begin{enumerate}
\item There is a canonical subgroup inclusion $\frac{(\mathcal{O}_{K}%
:J(U))}{\mathcal{O}_{K}}\hookrightarrow\operatorname*{Aut}(X)$, sending any
$\beta\in(\mathcal{O}_{K}:J(U))$ to the biholomorphism%
\begin{align*}
f:\mathbb{H}^{s}\times\mathbb{C}^{t}  &  \longrightarrow\mathbb{H}^{s}%
\times\mathbb{C}^{t}\\
\underline{z}  &  \longmapsto\underline{z}+\underline{\sigma}(\beta)\text{.}%
\end{align*}

\item There is a canonical subgroup inclusion $A_{U}\hookrightarrow
\operatorname*{Aut}(X)$.

\item The action of $\mathcal{O}_{K}^{\times,+}/U$.
\end{enumerate}
\end{lemma}

\begin{proof}
\textit{(1)} It is clear that this map is holomorphic and invertible on
$\mathbb{H}^{s}\times\mathbb{C}^{t}$. We need to show that it descends to the
quotient modulo $\mathcal{O}_{K}\rtimes U$. To this end, we need to check that
for all $u\in U$ and $\gamma\in\mathcal{O}_{K}$ the identity%
\begin{equation}
f(\underline{\sigma}(u)\underline{z}+\underline{\sigma}(\gamma))\equiv
f(\underline{z})\qquad\operatorname{mod}\qquad\mathcal{O}_{K}\rtimes
U\label{lvx1}%
\end{equation}
holds. Plugging in $f$ on the left hand side, we obtain $\underline{\sigma
}(u)\underline{z}+\underline{\sigma}(\gamma)+\underline{\sigma}(\beta
)\equiv\underline{z}+\underline{\sigma}(u^{-1})\underline{\sigma}%
(\gamma)+\underline{\sigma}(u^{-1})\underline{\sigma}(\beta)$ by letting
$\underline{\sigma}U$ act via $\underline{\sigma}(u^{-1})$,%
\[
=\underline{z}+\underline{\sigma}(u^{-1}\gamma)+\underline{\sigma}(u^{-1}%
\beta)+\underset{0}{\underbrace{\underline{\sigma}(\beta)-\underline{\sigma
}(\beta)}}=\underline{z}+\underline{\sigma}(u^{-1}\gamma)+\underline{\sigma
}((u^{-1}-1)\beta)+\underline{\sigma}(\beta)
\]
and since $u^{-1}\gamma\in\mathcal{O}_{K}$ as $u$ is a unit, as well as
$(u^{-1}-1)\beta\in\mathcal{O}_{K}$ by the very definition of the fractional
ideal $(\mathcal{O}_{K}:J(U))$, we can let $\underline{\sigma}\mathcal{O}_{K}$
act and obtain%
\[
\equiv\underline{z}+\underline{\sigma}(\beta)=f(\underline{z})\text{.}%
\]
This is exactly what we had to show, namely Equation \ref{lvx1}. Thus, $f$
descends to a biholomorphism $X\rightarrow X$. From deck transformation
theory, it follows that $f$ acts trivially on this quotient if and only if
$\underline{\sigma}(\beta)\in\underline{\sigma}\mathcal{O}_{K}$, so in total
we get a well-defined injection from the group $\frac{(\mathcal{O}_{K}%
:J(U))}{\mathcal{O}_{K}}$. This proves our first claim.\newline\textit{(2)} A
field automorphism $g\in A_{U}$ just maps elements to Galois conjugates, so at
worst it permutes the embeddings, say $\pi$ is given by $\sigma_{i}%
(g\beta)=\sigma_{\pi(i)}(\beta)$. Correspondingly, define%
\[
f(z_{1},\ldots,z_{s+t}):=f(z_{\pi(1)},\ldots,z_{\pi(s+t)})
\]
This is a biholomorphism. It descends modulo $\mathcal{O}_{K}\rtimes U$ since
a field automorphism maps $\mathcal{O}_{K}$ to itself, and by assumption we
have $gU=U$, so $U$ is also preserved.\newline\textit{(3)}\ Obvious.
\end{proof}

Consider a semi-direct product $G:=A\rtimes B$ with $A,B$ groups. Let
$\operatorname{Aut}(G;A)\subseteq\operatorname{Aut}(G)$ denote the subgroup of
automorphisms $\theta:G\rightarrow G$ such that $\theta(A)\subseteq A$, i.e.
those automorphisms which map the subgroup $A$ into itself.

We recall a result from group theory due to J. Dietz \cite{MR2981138}:\ There
is a canonical bijection between elements $\theta\in\operatorname{Aut}(G;A)$
and triples $(\alpha,\beta,\delta)$, where

\begin{itemize}
\item $\alpha\in\operatorname{Aut}(A)$,

\item $\delta\in\operatorname{Aut}(B)$,

\item $\beta\in\operatorname*{Map}(B,A)$,
\end{itemize}

such that the following conditions hold:

\begin{enumerate}
\item $\beta(b_{1}b_{2})=\beta(b_{1})\beta(b_{2})^{\delta(b_{1})}$ for all
$b_{1},b_{2}\in B$,

\item $\alpha(a^{b})=\alpha(a)^{\beta(b)\delta(b)}$ for all $a\in A$, $b\in B
$.
\end{enumerate}

We call such a triple $(\alpha,\beta,\delta)$ a \emph{Dietz triple}. This is
proven in \cite{MR2981138}, Lemma 2.1; we use the same notation as in the
paper to make it particularly easy to use the statement loc. cit. directly. In
her paper, Dietz writes a triple $(\alpha,\beta,\delta)$ as a matrix%
\[%
\begin{bmatrix}
\alpha & \beta\\
& \delta
\end{bmatrix}
\text{.}%
\]
To clarify notation, the superscripts in the conditions (1), (2) refer to the
conjugation%
\begin{equation}
g^{h}:=h^{-1}gh\label{lqqq1}%
\end{equation}
for arbitrary $g,h\in G$, and computed in $G$.

\begin{remark}
\label{rmk_Conjugate}We recall a basic fact: If $a\in A$ and $b\in B$, then in
the semi-direct product $A\rtimes B$ the conjugation $a^{b}$ agrees with the
action of $B$ on $A$ which underlies the semi-direct product structure.
\end{remark}

We apply these general remarks to the fundamental group of an OT\ manifold.

\begin{lemma}
\label{lemma_DietzForOT}There is a bijection between elements of
$\operatorname{Aut}(\pi)$ and triples $(\alpha,\beta,\delta)$ with

\begin{itemize}
\item $\alpha\in\operatorname{Aut}(\mathcal{O}_{K},+)$,

\item $\delta\in\operatorname{Aut}(U)$,

\item $\beta\in\operatorname*{Map}(U,\mathcal{O}_{K})$
\end{itemize}

such that the following conditions hold:

\begin{enumerate}
\item $\beta(b_{1}b_{2})=\beta(b_{1})+\beta(b_{2})\delta(b_{1})$ for all
$b_{1},b_{2}\in U$,

\item $\alpha(ab)=\alpha(a)\delta(b)$ for all $a\in\mathcal{O}_{K}$ and $b\in
U$.
\end{enumerate}
\end{lemma}

Here we use the notation \textquotedblleft$\operatorname{Aut}(\mathcal{O}%
_{K},+)$\textquotedblright\ to stress that we talk about the additive group
$(\mathcal{O}_{K},+)$, and not, as one could misunderstand, automorphisms of
$\mathcal{O}_{K}$ as a ring.

\begin{proof}
We wish to apply the above group-theoretical facts to the semi-direct product
$\pi:=\mathcal{O}_{K}\rtimes U$. This entails the following: (1) By Prop.
\ref{Prop_KappaAgreesWithOK} we have%
\[
1\longrightarrow(\mathcal{O}_{K},+)\longrightarrow\pi\longrightarrow
\pi_{\operatorname*{ab},\operatorname*{fr}}\longrightarrow1\text{.}%
\]
Every group automorphism $\theta:\pi\rightarrow\pi$ induces an automorphism of
the abelianization $\pi_{\operatorname*{ab}}$, and further on the torsion-free
quotient $\pi_{\operatorname*{ab},\operatorname*{fr}}$. Hence, by the above
exact sequence $\theta$ maps $(\mathcal{O}_{K},+)$ to itself. Hence,
$\operatorname{Aut}(\pi;(\mathcal{O}_{K},+))=\operatorname{Aut}(\pi)$ is an
equality of groups, i.e. we can describe arbitrary automorphisms using Dietz triples.

Working with the Dietz triples for $\pi$, conditions (1) and (2) unravel as follows:

\begin{enumerate}
\item $\beta(b_{1}b_{2})=\beta(b_{1})+\beta(b_{2})\delta(b_{1})$ for all
$b_{1},b_{2}\in U$,

\item $\alpha(ab)=\alpha(a)\delta(b)$ for all $a\in\mathcal{O}_{K}$ and $b\in
U$.
\end{enumerate}

We justify this: For (1) we write $(\mathcal{O}_{K},+)$ additively, giving
$\beta(b_{1}b_{2})=\beta(b_{1})+\beta(b_{2})^{\delta(b_{1})}$. Note that
$\beta(b_{2})\in\mathcal{O}_{K}$ and $\delta(b_{1})\in U$, so we may use
Remark \ref{rmk_Conjugate}\ to evaluate the conjugation $\beta(b_{2}%
)^{\delta(b_{1})}$ in $\pi$. Thus, $\beta(b_{2})^{\delta(b_{1})}$ is the
action of $\delta(b_{1})$ on $\beta(b_{2})$, but the semi-direct product
$\mathcal{O}_{K}\rtimes U$ is formed by letting $U$ act by multiplication on
$\mathcal{O}_{K}$, so this is simply the product $\beta(b_{2})\delta(b_{1})$
in the ring structure of $\mathcal{O}_{K}$. For (2) the original condition is%
\[
\alpha(a^{b})=\alpha(a)^{\beta(b)\delta(b)}\text{.}%
\]
Now, on the left side again $a\in\mathcal{O}_{K}$ while $b\in U$, so again by
Remark \ref{rmk_Conjugate} this is just the product $ab$ in the ring
$\mathcal{O}_{K}$. We have%
\[
\alpha(a)^{\beta(b)\delta(b)}=\left(  \left.  \alpha(a)^{\beta(b)}\right.
\right)  ^{\delta(b)}\text{.}%
\]
Here $\alpha(a)\in\mathcal{O}_{K}$ and $\beta(b)\in\mathcal{O}_{K}$, so we can
compute the conjugation within the group $\mathcal{O}_{K}$. Being abelian, the
conjugation is necessarily trivial. Thus, the expression simplifies to
$=\alpha(a)^{\delta(b)}$. Again, $\alpha(a)\in\mathcal{O}_{K}$ and
$\delta(b)\in U$, so by Remark \ref{rmk_Conjugate} this is just $\alpha
(a)\delta(b)$ in $\mathcal{O}_{K}$.
\end{proof}

\begin{lemma}
Suppose we are in the situation of the previous lemma. Then the automorphisms
corresponding to triples $(\alpha,\beta,\delta)$ with $\delta
:=\operatorname*{id}$ correspond to a subgroup of $\operatorname{Aut}(\pi)$
which is canonically isomorphic to%
\[
\left\{  \theta\in\operatorname{Aut}(\pi)\mid\delta=\operatorname*{id}%
\right\}  \cong(\mathcal{O}_{K}:J(U))\rtimes\operatorname{Aut}_{R}%
(\mathcal{O}_{K})\text{.}%
\]
Here $R$ is the smallest subring of $\mathcal{O}_{K}$ containing all $u\in U$,
and $\operatorname{Aut}_{R}(\mathcal{O}_{K})$ denotes the $R$-module
automorphisms of $\mathcal{O}_{K}$.
\end{lemma}

\begin{proof}
Assuming $\delta:=\operatorname*{id}$ the Dietz conditions become

\begin{enumerate}
\item $\beta(b_{1}b_{2})=\beta(b_{1})+b_{1}\beta(b_{2})$ for all $b_{1}%
,b_{2}\in U$,

\item $\alpha(ab)=\alpha(a)b$ for all $a\in\mathcal{O}_{K}$ and $b\in U$.
\end{enumerate}

Condition (2) means that $\alpha\in\operatorname{Aut}(\mathcal{O}_{K},+)$ is
not just an automorphism of $(\mathcal{O}_{K},+)$ as an abelian group, but as
an $R$-module over the subring $R\subseteq\mathcal{O}_{K}$ which is defined by
$R:=\mathbb{Z}[u\mid u\in U]$, i.e. the smallest subring of $\mathcal{O}_{K}$
containing all $u\in U$. We write $\alpha\in\operatorname{Aut}_{R}%
(\mathcal{O}_{K})$. Next, we use that $U$ is abelian. From $\beta(b_{2}%
b_{1})=\beta(b_{1}b_{2})$ and (1) we get%
\begin{align*}
\beta(b_{2})+b_{2}\beta(b_{1}) &  =\beta(b_{1})+b_{1}\beta(b_{2})\\
(b_{2}-1)\beta(b_{1}) &  =(b_{1}-1)\beta(b_{2})
\end{align*}
in the ring $\mathcal{O}_{K}$. Pick $b_{1}\in U\setminus\{1\}$ (exists!). Then
for all $b_{2}\in U\setminus\{1\}$ we obtain%
\[
\frac{\beta(b_{1})}{b_{1}-1}=\frac{\beta(b_{2})}{b_{2}-1}%
\]
in the fraction field $K$. Hence, this function is constant as $b_{2}$ varies
over $U\setminus\{1\}$. Let $c_{0}\in K$ be its value. Thus,%
\[
\beta(b)=c_{0}(b-1)
\]
holds for all $b\in U\setminus\{1\}$. Plugging in $b_{1}=b_{2}=1$ in the Dietz
condition (1), we also find $\beta(1)=0$, so this formula is actually valid
for all $b\in U$. Since $\beta(b)\in\mathcal{O}_{K}$ for all $b$ by
assumption, we deduce $c_{0}\in(\mathcal{O}_{K}:J(U))$, see\ Equation
\ref{tsyma1}. Recall that by Lemma \ref{Lemma_TranslationAction} for every
$c_{0}\in(\mathcal{O}_{K}:J(U))$ we in turn get an automorphism (in full
detail: get an biholomorphism of the OT\ manifold, which canonically induces
an automorphism of the fundamental group), so we have shown that there is a
left exact sequence%
\[
1\longrightarrow(\mathcal{O}_{K}:J(U))\longrightarrow\left\{  \theta
\in\operatorname{Aut}(\pi)\mid\delta=\operatorname*{id}\right\}
\longrightarrow\operatorname{Aut}_{R}(\mathcal{O}_{K})\text{,}%
\]
where we read the middle term as those automorphisms whose Dietz triple has
$\delta=\operatorname*{id}$. The left map is $c_{0}\mapsto(\operatorname*{id}%
,\beta,\operatorname*{id})$, where $\beta$ sends $b\mapsto c_{0}(b-1)$, and
the right map is $(\alpha,\beta,\operatorname*{id})\mapsto\alpha$. Indeed,
given any $\alpha\in\operatorname{Aut}_{R}(\mathcal{O}_{K})$ and defining
$\beta(b):=0$, we see that $(\alpha,\beta,\operatorname*{id})$ satisfies the
Dietz conditions. It follows that the above sequence is also exact on the
right and we leave it to the reader to check that this actually defines a
right section, so this is a split exact sequence. We obtain the semi-direct
product decomposition of our claim.
\end{proof}

We obtain a left exact sequence%
\begin{equation}
1\longrightarrow\left\{  \theta\in\operatorname{Aut}(\pi)\mid\delta
=\operatorname*{id}\right\}  \longrightarrow\operatorname{Aut}(\pi)\overset
{T}{\longrightarrow}\operatorname{Aut}(U)\text{,}\label{tsyma6}%
\end{equation}
where the left group corresponds to the triples $(\alpha,\beta
,\operatorname*{id}) $ and the right arrow $T$ is the map $(\alpha
,\beta,\delta)\mapsto\delta$.

\begin{lemma}
Suppose our OT manifold is of simple type. We have $\operatorname*{im}T=A_{U}%
$, where $A_{U}$ is as in Definition \ref{def_AKU}.
\end{lemma}

\begin{proof}
Let $(\alpha,\beta,\delta)$ be an arbitrary Dietz triple as in Lemma
\ref{lemma_DietzForOT}. Now, $\alpha\in\operatorname{Aut}(\mathcal{O}_{K},+)$.
Pick some $a\in\mathcal{O}_{K}$ such that $\alpha(a)\neq0$ (exists since
$\alpha$ is a bijection). Define a function $\varphi:U\rightarrow K$ by%
\begin{equation}
\varphi(b):=\frac{\alpha(ba)}{\alpha(a)}\qquad\text{for}\qquad b\in
U\text{.}\label{tsyma4}%
\end{equation}
By Dietz condition (2) we have $\alpha(ab)=\alpha(a)\delta(b)$, so this equals
$\delta(b)$. We note that the choice of $a$ irrelevant. We compute%
\[
\varphi(b_{1}b_{2})=\frac{\alpha(b_{1}b_{2}a)}{\alpha(a)}=\frac{\alpha
(b_{1}(b_{2}a))}{\alpha(b_{2}a)}\frac{\alpha(b_{2}a)}{\alpha(a)}\text{,}%
\]
but $\frac{\alpha(b_{1}(b_{2}a))}{\alpha(b_{2}a)}=\varphi(b_{1})$ since, as we
had explained, the choice of $a$ is irrelevant, so we could also take $b_{2}a$
instead (moreover, $\alpha(b_{2}a)=\delta(b_{2})\alpha(a)$ by condition (2)
and since $\delta$ takes values in $U$, $\alpha(a)\neq0$ implies that
$\alpha(b_{2}a)\neq0$, so the division above was fine).\ Thus, we find%
\[
\varphi(b_{1}b_{2})=\varphi(b_{1})\cdot\varphi(b_{2})
\]
for all $b_{1},b_{2}\in U$. Similarly, one checks that $\varphi(b_{1}%
+b_{2})=\varphi(b_{1})+\varphi(b_{2})$. Thus, by linear extension, we obtain
that $\varphi:U\rightarrow K$ can be extended to a ring homomorphism%
\[
\varphi:R\longrightarrow K\text{,}%
\]
where $R$ is the smallest subring of $\mathcal{O}_{K}$ containing all $u\in U
$ as before. As $X$ is by assumption of simple type, there is no proper
subfield of $K$ which already contains $U$. Thus, the field of fractions of
$R$, which by $R\subseteq\mathcal{O}_{K}$ is contained in $K$, must be $K$
itself. Hence, $\varphi$, by extension to the field of fractions
$\varphi(x/y):=\varphi(x)/\varphi(y)$ defines a field automorphism
$\varphi:K\rightarrow K$. As we had remarked below Equation \ref{tsyma4},
$\varphi\mid_{U}=\delta$, but $\delta\in\operatorname{Aut}(U)$, so $\varphi
U\subseteq U$. It follows $\varphi\in A_{U}$.
\end{proof}

\begin{lemma}
Suppose our OT manifold is of simple type. Then for $\pi:=\pi_{1}(X)$,
$\operatorname{Aut}(\pi)$ is canonically isomorphic to%
\[
\left\{  \theta\in\operatorname{Aut}(\pi)\mid\delta=\operatorname*{id}%
\right\}  \rtimes A_{U}\text{.}%
\]

\end{lemma}

\begin{proof}
By the previous lemma and Equation \ref{tsyma6}, we have the exact sequence%
\[
1\longrightarrow\left\{  \theta\in\operatorname{Aut}(\pi)\mid\delta
=\operatorname*{id}\right\}  \longrightarrow\operatorname{Aut}(\pi)\overset
{T}{\longrightarrow}A_{U}\text{.}%
\]
A right splitting is given by sending $\varphi\in A_{U}$ to $(\varphi
\mid_{\mathcal{O}_{K}},0,\varphi\mid_{U})$. The Dietz conditions are easily
seen to hold.
\end{proof}

\begin{lemma}
Suppose our OT\ manifold is of simple type. Then $\operatorname*{Aut}%
\nolimits_{R}(\mathcal{O}_{K})=\mathcal{O}_{K}^{\times}$, where $R$ is the
smallest subring of $\mathcal{O}_{K}$ containing all $u\in U$.
\end{lemma}

\begin{proof}
Suppose $g\in\operatorname*{Aut}\nolimits_{R}(\mathcal{O}_{K})$. Let
$\beta,\lambda\in\mathcal{O}_{K}$ be arbitrary. As $X$ is of simple type, we
have $\mathbb{Q}\cdot R=K$, i.e. $\beta=\frac{1}{n}r$ for some $n\geq1$ and
$r\in R$. Then $g(\beta\lambda)=g(\frac{1}{n}r\lambda)=\frac{1}{n}rg(\lambda)
$, as $g$ is an $R$-module homomorphism. Hence, $g(\beta\lambda)=\beta
g(\lambda)$. It follows that $g$ is even an $\mathcal{O}_{K}$-module
homomorphism. Thus, $g\in\operatorname*{Aut}\nolimits_{\mathcal{O}_{K}%
}(\mathcal{O}_{K})$ and since $\mathcal{O}_{K}$ is free of rank one over
itself, $\operatorname*{Aut}\nolimits_{\mathcal{O}_{K}}(\mathcal{O}_{K}%
)\cong\mathcal{O}_{K}^{\times}$; the converse inclusion is obvious.
\end{proof}

Combining the previous lemmas, we obtain the following result.

\begin{proposition}
\label{prop_AutPi}Suppose our OT manifold $X$ is of simple type. Then for
$\pi:=\pi_{1}(X)$, the automorphism group $\operatorname{Aut}(\pi)$ is
canonically isomorphic to%
\[
\left(  \left.  (\mathcal{O}_{K}:J(U))\rtimes\mathcal{O}_{K}^{\times}\right.
\right)  \rtimes A_{U}\text{.}%
\]
More concretely, it has a canonical ascending three step filtration
$F_{0}\subseteq F_{1}\subseteq F_{2}$ whose graded pieces are%
\begin{align*}
\operatorname{gr}_{0}^{F}\operatorname{Aut}(\pi) &  \cong(\mathcal{O}%
_{K}:J(U))\text{;}\\
\operatorname{gr}_{1}^{F}\operatorname{Aut}(\pi) &  \cong\mathcal{O}%
_{K}^{\times}\text{;}\\
\operatorname{gr}_{2}^{F}\operatorname{Aut}(\pi) &  \cong A_{U}\text{.}%
\end{align*}
These isomorphisms are all canonical.
\end{proposition}

Now we are ready to prove the key ingredient for our results.

\begin{theorem}
\label{thm_Aut}Suppose our OT manifold is of simple type. Then the
biholomorphism group $\operatorname{Aut}(X)$ is canonically isomorphic to%
\begin{equation}
\left(  \left.  \left(  \frac{(\mathcal{O}_{K}:J(U))}{\mathcal{O}_{K}}\right)
\rtimes\left(  \frac{\mathcal{O}_{K}^{\times,+}}{U}\right)  \right.  \right)
\rtimes A_{U}\text{.}\label{l3}%
\end{equation}
More concretely, it has a canonical ascending three step filtration
$F_{0}\subseteq F_{1}\subseteq F_{2}$ whose graded pieces are%
\begin{align*}
\operatorname{gr}_{0}^{F}\operatorname{Aut}(X) &  \simeq\mathcal{O}%
_{K}/J(U)\text{;}\\
\operatorname{gr}_{1}^{F}\operatorname{Aut}(X) &  \cong\mathcal{O}_{K}%
^{\times,+}/U\text{;}\\
\operatorname{gr}_{2}^{F}\operatorname{Aut}(X) &  \cong A_{U}\text{.}%
\end{align*}
For $\operatorname{gr}_{0}^{F}\operatorname{Aut}(X)$ the isomorphism is
non-canonical, while for $\operatorname{gr}_{i}^{F}\operatorname{Aut}(X)$ with
$i=1,2$ it is.
\end{theorem}

\begin{proof}
By Lemma \ref{Lemma_TranslationAction} all three groups in Equation \ref{l3}
indeed induce biholomorphisms, but jointly they generate the entire iterated
semi-direct product, so we just have to show that there are no other
biholomorphisms. Let $\theta:X\rightarrow X$ be an arbitrary biholomorphism.
It lifts to the universal covering space,%
\[
\tilde{\theta}:\mathbb{H}^{s}\times\mathbb{C}^{t}\longrightarrow\mathbb{H}%
^{s}\times\mathbb{C}^{t}\text{.}%
\]
Moreover, it induces a canonical map $\theta_{\ast}:\pi_{1}(X,\ast
)\rightarrow\pi_{1}(X,\ast)$ on the fundamental group, and by Prop.
\ref{prop_AutPi} we get an element in%
\[
\left(  \left.  (\mathcal{O}_{K}:J(U))\rtimes\mathcal{O}_{K}^{\times}\right.
\right)  \rtimes A_{U}\text{.}%
\]
We leave it to the reader to check that we can write $\mathcal{O}_{K}%
^{\times,+}$ instead of $\mathcal{O}_{K}^{\times}$, which amounts to the fact
that $\tilde{\theta}$ preserves being in the upper half plane. Now, by Lemma
\ref{Lemma_TranslationAction} we may associate a (possibly different)
biholomorphism $\theta^{\prime}$ to this element. Thus, we learn that
$f:=\theta\theta^{\prime-1}$ is a biholomorphism of $X$ which induces the
identity on $\pi_{1}(X,\ast)$. We are done once we prove that
$f=\operatorname*{id}$. Firstly, also $f$ lifts to an automorphism $\tilde{f}$
of the universal covering space. Since $\tilde{f}$ descends modulo the action
of $\mathcal{O}_{K}$, we deduce that for any $\gamma\in\mathcal{O}_{K}$ and
$\underline{z}=(z_{1},\ldots,z_{s+t})\in\mathbb{H}^{s}\times\mathbb{C}^{t}$
there exists some $\gamma_{z}^{\prime}\in\mathcal{O}_{K}$ such that%
\begin{equation}
\tilde{f}(\underline{z}+\underline{\sigma}(\gamma))-\tilde{f}(\underline
{z})=\underline{\sigma}(\gamma_{\underline{z}}^{\prime})\text{.}\label{lv1}%
\end{equation}
If we fix $\gamma$ and let the point $\underline{z}$ vary, the value of
$\gamma_{\underline{z}}^{\prime}$ must vary continuously in $\underline{z}$.
Since the image of $\underline{\sigma}$ is discrete, it follows that this
function is locally constant and since $\mathbb{H}^{s}\times\mathbb{C}^{t}$ is
connected, it must be constant in $\underline{z}$. Then taking derivatives of
Equation \ref{lv1} yields%
\[
\frac{\partial\tilde{f}}{\partial z_{i}}(\underline{z}+\sigma(\gamma
))=\frac{\partial\tilde{f}}{\partial z_{i}}(\underline{z})\text{.}%
\]
It follows that the partial derivatives $\frac{\partial\tilde{f}}{\partial
z_{i}}$ descend to the quotient $(\mathbb{H}^{s}\times\mathbb{C}%
^{t})/\underline{\sigma}(\mathcal{O}_{K})$. However, $(\mathbb{H}^{s}%
\times\mathbb{C}^{t})/\underline{\sigma}(\mathcal{O}_{K})$ is an example of a
Cousin group, as was proven by Oeljeklaus and Toma \cite[Lemma 2.4]{MR2141693}
(this is also discussed in \cite{MR3326586}, \cite{MR3341439}), it carries no
holomorphic functions except the constant ones. Thus, these partial
derivatives are necessarily constant. It follows that%
\begin{equation}
\tilde{f}(\underline{z})=A\underline{z}+B\label{lv1a}%
\end{equation}
for a matrix $A$. As $\tilde{f}$ induces the identity on $\pi_{1}$, it follows
that for any $u\in U$ and any $a\in\mathcal{O}_{K}$ we have%

\[
\tilde{f}(\sigma(u)\underline{z}+\sigma(a))=\sigma(u)\tilde{f}(\underline
{z})+\sigma(a).
\]
We hence get $A\sigma(a)+B=\sigma(u)B+\sigma(a)$ for all $u\in U,a\in
\mathcal{O}_{K}.$ But this plainly implies $A=\operatorname{id},B=0$.
\end{proof}

\section{\label{sect_RCFT}Review of some class field theory}

We briefly recall the (very few!) tools we need from class field theory. Let
$K$ be a number field. A \emph{modulus} for $K$ is a function%
\[
\mathfrak{m}:\{\text{places of the number field }K\}\longrightarrow
\mathbb{Z}_{\geq0}%
\]
such that (1) for all but finitely many places $P$ we have $\mathfrak{m}%
(P)=0$, (2) if $P$ is a real place, we only allow $\mathfrak{m}(P)\in\{0,1\}$,
and (3) for complex places $P$ we demand $\mathfrak{m}(P)=0$. The
algebro-geometrically inclined reader might prefer to think of a modulus as an
effective Weil divisor on%
\[
\operatorname*{Spec}(\mathcal{O}_{K})\cup\{\text{real places}\}\text{,}%
\]
where real places are only allowed to have multiplicity zero or one. Fitting
into this pattern, let $\mathfrak{m}_{0}\subseteq\mathcal{O}_{K}$ be the ideal
defined by the prime factorization%
\[
\mathfrak{m}_{0}=\prod P^{\mathfrak{m}(P)}\text{,}%
\]
i.e. literally we take the possibly non-reduced closed subscheme cut out by
the Weil divisor, ignoring the datum at the real places. One customarily also
says that an ideal $I$ divides $\mathfrak{m}$ if we have $\mathfrak{m}_{0}\mid
I$ as ideals in $\mathcal{O}_{K}$.

There is the standard group homomorphism%
\begin{equation}
\operatorname*{div}:K^{\times}\longrightarrow\coprod_{P\in\left(
\text{maximal ideals of }\mathcal{O}_{K}\right)  }\mathbb{Z}\text{,}\qquad
a\longmapsto(v_{P}(a))_{P}\text{,}\label{lci1}%
\end{equation}
which associates to any element $a\in K^{\times}$ the exponents $v_{P}(a)$ of
its unique prime ideal factorization, $P$ being one of the maximal primes.
Equivalently, this is the map sending a rational function on
$\operatorname*{Spec}\mathcal{O}_{K}$ to its Weil divisor. There is a slight
variation of this theme:

\begin{definition}
\label{Def_VariantsWithModulus}For $K$ a number field and $\mathfrak{m}$ a
modulus, define

\begin{enumerate}
\item $I^{S(\mathfrak{m})}:=\coprod_{P,\mathfrak{m}(P)=0}\mathbb{Z}$, where
$P$ runs through the prime ideals of $\mathcal{O}_{K}$; or equivalently this
is the group of Weil divisors of $\operatorname*{Spec}(\mathcal{O}_{K}%
)-\{$primes dividing $\mathfrak{m}\}$.

\item $K_{\mathfrak{m},1}:=\{a\in K^{\times}\mid v_{P}(a-1)\geq\mathfrak{m}%
(P)$ for all $P\mid\mathfrak{m}$, and moreover $\sigma(a)>0$ for all real
embeddings with $\mathfrak{m}(\sigma)=1\}$,

\item $U_{\mathfrak{m},1}:=K_{\mathfrak{m},1}\cap\mathcal{O}_{K}^{\times}$.
\end{enumerate}
\end{definition}

Once we pick an arbitrary modulus $\mathfrak{m}$, we can refine Equation
\ref{lci1}, in the obvious way, to a group homomorphism%
\[
K_{\mathfrak{m},1}\longrightarrow I^{S(\mathfrak{m})}\text{.}%
\]
If $\mathfrak{m}=1$ is the zero modulus, i.e. $\mathfrak{m}(P)=0$ for all
places $P$, this becomes Equation \ref{lci1}.

\begin{definition}
\label{Def_RayClassGroup}For an arbitrary modulus $\mathfrak{m}$ we call%
\[
C_{\mathfrak{m}}:=I^{S(\mathfrak{m})}/K_{\mathfrak{m},1}%
\]
the \emph{ray class group} modulo $\mathfrak{m}$.
\end{definition}

\begin{theorem}
[Global Class Field Theory]Let $K$ be a number field.

\begin{enumerate}
\item For every modulus $\mathfrak{m}$, the ray class group $C_{\mathfrak{m}}
$ is finite, and there exists a canonical finite abelian field extension
$L_{\mathfrak{m}}/K$ along with a canonical group isomorphism%
\[
\psi_{K,\mathfrak{m}}:C_{\mathfrak{m}}\overset{\sim}{\longrightarrow
}\operatorname*{Gal}(L_{\mathfrak{m}}/K)\text{.}%
\]
The field $L_{\mathfrak{m}}$ is known as the \emph{ray class field} of
$\mathfrak{m}$.

\item In fact, $L_{\mathfrak{m}}$ can be characterized uniquely as the largest
abelian field extension of $K$ such that the ramification of $L_{\mathfrak{m}%
}$ over $K$ is bounded from above by the multiplicities of $\mathfrak{m}$. The
multiplicity $0$ or $1$ at the real places means whether we allow a real place
to split into a pair of complex conjugate embeddings in $L_{\mathfrak{m}}$
(multiplicity $1$) or demand it to stay real (multiplicity $0$).

\item If $\mathfrak{m}\leq\mathfrak{m}^{\prime}$ this induces an
order-reversing correspondence $L_{\mathfrak{m}}\subseteq L_{\mathfrak{m}%
^{\prime}}$ and the diagram%
\[%
\begin{array}
[c]{rccc}%
\psi_{K,\mathfrak{m}^{\prime}}: & C_{\mathfrak{m}^{\prime}} & \overset{\sim
}{\longrightarrow} & \operatorname*{Gal}(L_{\mathfrak{m}^{\prime}}/K)\\
& \downarrow &  & \downarrow\\
\psi_{K,\mathfrak{m}}: & C_{\mathfrak{m}} & \overset{\sim}{\longrightarrow} &
\operatorname*{Gal}(L_{\mathfrak{m}}/K)
\end{array}
\]
commutes. Here the left-hand side downward arrow is the natural surjection
from changing $\mathfrak{m}$ in Definition \ref{Def_RayClassGroup}, while the
right-hand side downward arrow comes from the Galois tower%
\[%
\begin{array}
[c]{c}%
L_{\mathfrak{m}^{\prime}}\\
\mid\\
L_{\mathfrak{m}}\\
\mid\\
K
\end{array}
\]

\end{enumerate}
\end{theorem}

\subsection{\label{section_SpecialModuli}Exceptional moduli}

\begin{lemma}
\label{Lemma_EasyInclusion}Let $\mathfrak{m}$ be a modulus. Then
$J(U_{\mathfrak{m},1})\subseteq\mathfrak{m}_{0}$.
\end{lemma}

\begin{proof}
Every element in $J(U_{\mathfrak{m},1})$ is of the shape $a=\sum a_{i}%
(u_{i}-1)$ for $a_{i}\in\mathcal{O}_{K}$ and $u_{i}\in U_{\mathfrak{m},1}$.
The unique prime ideal factorization of $\mathfrak{m}_{0}$ is (by the very
definition of $\mathfrak{m}_{0}$), $\mathfrak{m}_{0}=\prod P^{\mathfrak{m}%
(P)}$, and so it suffices to check that $v_{P}(a)\geq\mathfrak{m}(P)$ for all
prime ideals $P$. If $P$ divides $\mathfrak{m}$, we have%
\begin{equation}
v_{P}(u_{i}-1)\geq\mathfrak{m}(P)\label{lza1}%
\end{equation}
for all $u_{i}\in U_{\mathfrak{m},1}$, just by Definition
\ref{Def_VariantsWithModulus}, so by the ultrametric inequality for
valuations, we find%
\[
v_{P}(a)\geq\min\left\{  v_{P}(a_{i}(u_{i}-1))\right\}  \geq\min\left\{
v_{P}(u_{i}-1)\right\}  \geq\mathfrak{m}(P)\text{,}%
\]
so this is fine. If $P$ does not divide $\mathfrak{m}$, there is no
counterpart of the condition of Equation \ref{lza1} in the definition of
$U_{\mathfrak{m},1}$, so we just get $v_{P}(u_{i}-1)\geq0$ since $u_{i}%
\in\mathcal{O}_{K}^{\times}$ and therefore $u_{i}-1\in\mathcal{O}_{K}$ is
integral. On the other hand, then $\mathfrak{m}(P)=0$, so actually Equation
\ref{lza1} holds simply for \textit{all} prime ideals $P$.
\end{proof}

The following definition goes in the direction of a sufficient criterion to
have equality:

\begin{definition}
Let $K$ be a number field. We say that the modulus $\mathfrak{m}$ is
\emph{exceptional} if

\begin{enumerate}
\item it has $\mathfrak{m}(P)=1$ for all real places, and

\item the ideal $\mathfrak{m}_{0}$ admits a set of generators $g_{1}%
,\ldots,g_{r}$ such that each $g_{i}+1$ is a totally positive unit, i.e. an
element of $\mathcal{O}_{K}^{\times,+}$.
\end{enumerate}
\end{definition}

\begin{lemma}
\label{Lemma_EasyInclusionConverse}If $\mathfrak{m}$ is an exceptional
modulus, we have equality of ideals $J(U_{\mathfrak{m},1})=\mathfrak{m}_{0}$.
\end{lemma}

\begin{proof}
The inclusion $J(U_{\mathfrak{m},1})\subseteq\mathfrak{m}_{0}$ is just Lemma
\ref{Lemma_EasyInclusion}. We show the converse $\mathfrak{m}_{0}\subseteq
J(U_{\mathfrak{m},1})$: Suppose $g\in\mathfrak{m}_{0}$. Then \textit{if} $g+1
$ happens to be a totally positive unit, we get%
\[
v_{P}(\underset{\in\mathcal{O}_{K}^{\times,+}}{(g+1)}-1)=v_{P}(g)\geq
\mathfrak{m}(P)
\]
for all prime ideals $P$, and moreover $\sigma(g+1)>0$ for all the real places
$\sigma$. So in this case, we indeed have $g+1\in U_{\mathfrak{m},1}$. Thus,
for an arbitrary $a\in\mathfrak{m}_{0}$, we expand it in terms of the ideal
generators%
\[
a=\sum a_{i}g_{i}=\sum a_{i}(\underset{\in U_{\mathfrak{m},1}}{\underbrace
{(g_{i}+1)}}-1)\in J(U_{\mathfrak{m},1})\text{.}%
\]

\end{proof}

Let us discuss a little how to work with exceptional moduli:

\begin{example}
Suppose $\mathfrak{m}$ is a given modulus with $\mathfrak{m}(P)=1$ for all
real places and we want to check whether it is exceptional. To this end,
compute the ray unit group $U_{\mathfrak{m},1}$. If $J(U_{\mathfrak{m},1}%
)\neq\mathfrak{m}_{0}$, then $\mathfrak{m}$ is not exceptional because
otherwise this would contradict Lemma \ref{Lemma_EasyInclusionConverse}.
Conversely, if $J(U_{\mathfrak{m},1})=\mathfrak{m}_{0}$, then $\mathfrak{m}$
is exceptional since the ideal $J$ by its very definition is indeed generated
from units $g_{i}$ such that $g_{i}+1\in U_{\mathfrak{m},1}$ and
$U_{\mathfrak{m},1}\subseteq\mathcal{O}_{K}^{\times,+}$ by our condition on
the real places.
\end{example}

\begin{example}
Of course, computing $J(U_{\mathfrak{m},1})$ is costly, so for explicit
example cases of exceptional moduli, the approach of the previous example is
not to be recommended. Much better, one should simply pick a finite index
subgroup $U\subseteq\mathcal{O}_{K}^{\times,+}$ and right away work with
$\mathfrak{m}_{0}:=J(U)$, and $\mathfrak{m}(P)=1$ for all real places. Then
$\mathfrak{m}$ is an exceptional modulus by construction. We may consider this
strategy for the following family \cite[\S 9]{otarith}: Suppose $m\geq1 $.
Then the polynomial%
\[
f(T)=T^{3}+mT-1
\]
is irreducible, generates a cubic number field $K$ with one real and one
complex place, and the image of $T$ in the number field, which we denote by
$u:=\overline{T}$, is a totally positive unit. Take $U_{l}:=\left\langle
u^{l}\right\rangle $. Now, one needs to compute the fundamental unit $v$ of
$K$ so that%
\[
\mathcal{O}_{K}^{\times}=\left\langle -1\right\rangle \times\left\langle
v\right\rangle \qquad\text{and}\qquad\mathcal{O}_{K}^{\times,+}=\left\langle
1\right\rangle \times\left\langle v\right\rangle \text{,}%
\]
i.e. $\mathcal{O}_{K}^{\times}/\mathcal{O}_{K}^{\times,+}\simeq\{\pm1\}$.
Define an exceptional modulus $\mathfrak{m}$ via $\mathfrak{m}_{0}:=J(U_{l})$.
It follows that $J(U_{\mathfrak{m},1})=J(U_{l})$. In a single computation, one
finds the exponent $e$ in $u=v^{e}$, and then $U/U_{\mathfrak{m},1}%
=\{\pm1\}\times\mathbb{Z}/(le\mathbb{Z})$, so that $\#U/U_{\mathfrak{m}%
,1}=2le$. We see that this produces an infinite family of exceptional moduli.
\end{example}

\section{\label{sect_TorsionHomologyRayClassGroups}Torsion homology and ray
class groups}

Next, we need the following important computation from classical class field theory:

\begin{proposition}
\label{Prop_BasicRayClassFieldSequence}For $K$ an arbitrary number field and
$\mathfrak{m}$ an arbitrary modulus such that $\mathfrak{m}(P)=1$ for all real
places, there is an exact sequence of abelian groups%
\[
1\longrightarrow\frac{\mathcal{O}_{K}^{\times,+}}{U_{\mathfrak{m},1}%
}\longrightarrow\left(  \mathcal{O}_{K}/\mathfrak{m}_{0}\right)  ^{\times
}\longrightarrow C_{\mathfrak{m}}\longrightarrow C\longrightarrow0\text{.}%
\]
Here $C$ denotes the ordinary ideal class group (= $C_{0}$, the ray class
group for the trivial modulus).
\end{proposition}

\begin{proof}
This is \cite[Ch. VI, \S 1, Exercise 13]{MR1697859}. This exercise follows
directly from \cite[Ch. VI, \S 1, (1.11)\ Prop.]{MR1697859}.
\end{proof}

The cardinalities $h_{\mathfrak{m}}:=\left\vert C_{\mathfrak{m}}\right\vert $
(and same for the trivial modulus, $h:=\left\vert C\right\vert $) are known as
the \emph{ray class number} (resp. \emph{class number}).

\begin{theorem}
\label{thm_m copy}Let $K$ be a number field with $s\geq1$ real places and
precisely one complex place. Moreover, suppose $\mathfrak{m}$ is an
exceptional modulus. Then $U_{\mathfrak{m},1}$ is an admissible subgroup in
the sense of \cite[\S 1]{MR2141693}. Let $X:=X(K,U_{\mathfrak{m},1})$ be the
corresponding Oeljeklaus-Toma manifold. Then the graded Euler characteristic%
\[
\chi^{F}(\operatorname{Aut}(X))=\prod_{i}\left\vert \operatorname{gr}_{i}%
^{F}\operatorname{Aut}(X)\right\vert ^{(-1)^{i}}%
\]
satisfies%
\[
\frac{h_{\mathfrak{m}}}{h}\leq\frac{\chi^{F}(\operatorname{Aut}(X))}%
{\left\vert A_{U_{\mathfrak{m},1}}\right\vert }\text{,}%
\]
where $h_{\mathfrak{m}}$ denotes the ray class number of $\mathfrak{m}$ and
$h$ is the ordinary class number.
\end{theorem}

\begin{proof}
We begin with the $4$-term exact sequence of Prop.
\ref{Prop_BasicRayClassFieldSequence}. Since $\mathfrak{m}$ is a exceptional
modulus, by Lemma \ref{Lemma_EasyInclusionConverse} we have $J(U_{\mathfrak{m}%
,1})=\mathfrak{m}_{0}$, so this sequence specializes to%
\begin{equation}
1\longrightarrow\frac{\mathcal{O}_{K}^{\times,+}}{U_{\mathfrak{m},1}%
}\longrightarrow\left(  \frac{\mathcal{O}_{K}}{J(U_{\mathfrak{m},1})}\right)
^{\times}\longrightarrow\ker\left(  C_{\mathfrak{m}}\twoheadrightarrow
C\right)  \longrightarrow0\text{.}\label{ags2}%
\end{equation}
Although there are much more direct ways to show this, note that this implies
that $U/U_{\mathfrak{m},1}$ is finite. In particular, the free rank of
$U_{\mathfrak{m},1}$ agrees with the one of $U=\mathcal{O}_{K}^{\times}$, and
so is $s$ by Dirichlet's Unit Theorem. Moreover, $U_{\mathfrak{m},1}%
\subseteq\mathcal{O}_{K}^{\times,+}$ lies in the subgroup of totally positive
units thanks to our condition on the real places in the modulus. It follows
that $U_{\mathfrak{m},1}$ is admissible in the sense of Oeljeklaus and Toma.
Next, class field theory for the trivial modulus as well as $\mathfrak{m}$
produces the tower of class fields%
\[%
\begin{array}
[c]{ccl}%
L_{\mathfrak{m}} & \quad & \text{ray class field for }\mathfrak{m}\\
\mid &  & \\
H & \quad & \text{Hilbert class field}\\
\mid &  & \\
K &  &
\end{array}
\]
so that the Artin reciprocity symbol provides us with canonical and natural
isomorphisms%
\[
\operatorname*{Gal}(L_{\mathfrak{m}}/K)\cong C_{\mathfrak{m}}\qquad
\text{and}\qquad\operatorname*{Gal}(H/K)\cong C\text{.}%
\]
Thus, we have $\operatorname*{Gal}(L_{\mathfrak{m}}/H)\cong\ker
(C_{\mathfrak{m}}\twoheadrightarrow C)$; and moreover by the tower law of
field extension degrees, $h_{\mathfrak{m}}=\left\vert \operatorname*{Gal}%
(L_{\mathfrak{m}}/K)\right\vert \cdot h$. We obtain the first and second
equalities in the following equation, and the third follows from the exactness
of Sequence \ref{ags2}:%
\begin{equation}
\frac{h_{\mathfrak{m}}}{h}=\left\vert \ker(C_{\mathfrak{m}}\twoheadrightarrow
C)\right\vert =\left\vert \operatorname*{Gal}(L_{\mathfrak{m}}/H)\right\vert
=\frac{\left\vert \left(  \frac{\mathcal{O}_{K}}{J(U_{\mathfrak{m},1}%
)}\right)  ^{\times}\right\vert }{\left\vert \frac{\mathcal{O}_{K}^{\times,+}%
}{U_{\mathfrak{m},1}}\right\vert }\text{.}\label{lcio5a}%
\end{equation}
By Theorem \ref{thm_Aut} we have a canonical filtration of the biholomorphism
group,%
\begin{align}
\operatorname{gr}_{0}^{F}\operatorname{Aut}(X) &  \simeq\mathcal{O}%
_{K}/J(U_{\mathfrak{m},1})\text{;}\nonumber\\
\operatorname{gr}_{1}^{F}\operatorname{Aut}(X) &  \cong\mathcal{O}_{K}%
^{\times,+}/U_{\mathfrak{m},1}\text{;}\label{lcio5}\\
\operatorname{gr}_{2}^{F}\operatorname{Aut}(X) &  \cong A_{U_{\mathfrak{m},1}%
}\text{.}\nonumber
\end{align}
Thus, if we form a type of multiplicative Euler characteristic along the
graded pieces%
\[
\chi^{F}(\operatorname{Aut}(X)):=\prod_{i}\left\vert \operatorname{gr}_{i}%
^{F}\operatorname{Aut}(X)\right\vert ^{(-1)^{i}}=\frac{\left\vert
\mathcal{O}_{K}/J(U_{\mathfrak{m},1})\right\vert \cdot\left\vert
A_{U_{\mathfrak{m},1}}\right\vert }{\left\vert \mathcal{O}_{K}^{\times
,+}/U_{\mathfrak{m},1}\right\vert }\text{,}%
\]
we deduce from Equation \ref{lcio5a} that%
\[
\frac{h_{\mathfrak{m}}}{h}=\frac{\left\vert \left(  \frac{\mathcal{O}_{K}%
}{J(U_{\mathfrak{m},1})}\right)  ^{\times}\right\vert }{\left\vert
\frac{\mathcal{O}_{K}^{\times,+}}{U_{\mathfrak{m},1}}\right\vert }\leq
\frac{\left\vert \frac{\mathcal{O}_{K}}{J(U_{\mathfrak{m},1})}\right\vert
}{\left\vert \frac{\mathcal{O}_{K}^{\times,+}}{U_{\mathfrak{m},1}}\right\vert
}=\frac{\chi^{F}(\operatorname{Aut}(X))}{\left\vert A_{U_{\mathfrak{m},1}%
}\right\vert }\text{.}%
\]
This finishes the proof.
\end{proof}

In a way, the principal point we wish to call attention to is that the
so-called ray class group of a modulus $\mathfrak{m}$, or the Galois group
which is associated to it by class field theory, sits in a canonical exact
sequence%
\begin{equation}
1\longrightarrow\frac{\mathcal{O}_{K}^{\times,+}}{U_{\mathfrak{m},1}%
}\longrightarrow\left(  \frac{\mathcal{O}_{K}}{J(U_{\mathfrak{m},1})}\right)
^{\times}\longrightarrow\operatorname*{Gal}(L_{\mathfrak{m}}/H)\longrightarrow
0\text{,}\label{labi1}%
\end{equation}
while (as we have shown) the automorphism group of $X(K,U)$ possesses a
canonical filtration $F_{\bullet}$ whose graded pieces are (non-canonically)
isomorphic to the groups in Equation \ref{lcio5}. The group
$A_{U_{\mathfrak{m},1}}$ will frequently be trivial. Whenever this happens,
note that the Sequence \ref{labi1} could, albeit with quite some abuse of
language, be rewritten as%
\[
\text{\textquotedblleft}1\longrightarrow\operatorname{gr}_{1}^{F}%
\operatorname{Aut}(X)\longrightarrow\left(  \operatorname{gr}_{0}%
^{F}\operatorname{Aut}(X)\right)  ^{\times}\longrightarrow\operatorname*{Gal}%
(L_{\mathfrak{m}}/H)\longrightarrow0\text{\textquotedblright.}%
\]

\section{Exponential torsion asymptotics}

Finally, in the case of Oeljeklaus--Toma surfaces, the homology torsion growth
can be related to the Mahler measure of a minimal polynomial.

\begin{theorem}
\label{thm_TorsionMahler}Let $K$ be a number field with $s=t=1$ embeddings and
$u$ a (totally) positive fundamental unit, i.e. $\mathcal{O}_{K}^{\times
,+}\simeq\mathbb{Z}\left\langle u\right\rangle $. Then all groups%
\[
U_{n}:=\mathbb{Z}\left\langle u^{n}\right\rangle
\]
are admissible, and for $X_{n}:=X(K,U_{n})$ we have%
\begin{equation}
\lim_{n\longrightarrow\infty}\frac{\log\left\vert H_{1}(X_{n},\mathbb{Z}%
)_{\operatorname*{tor}}\right\vert }{n}=\log M(f)\text{,}\label{lcd1}%
\end{equation}
where $f$ is the minimal polynomial of the unit $u$, and $M(f)$ denotes the
Mahler measure of $f$. Hence, $\left\vert H_{1}(X_{n},\mathbb{Z}%
)_{\operatorname*{tor}}\right\vert $ always grows asymptotically exponentially
as $n\rightarrow+\infty$.
\end{theorem}

Note that in this case each $X(K,U_{n})$ is an Inoue surface $X$ of type
$S^{0}$. Further, by Dirichlet's Unit Theorem, $\mathcal{O}_{K}^{\times}%
\simeq\left\langle -1\right\rangle \times\mathbb{Z}\left\langle u\right\rangle
$ with $u$ any fundamental unit. Thus, either $u$ is totally positive so that
$\mathcal{O}_{K}^{\times,+}\simeq\mathbb{Z}\left\langle u\right\rangle $, or
otherwise this is true after replacing $u$ by $-u$. Hence, once we have
$s=t=1$, a choice of $u$ as in the statement of the theorem is always possible.

\begin{proof}
By Prop. \ref{Prop_KappaAgreesWithOK} we have $\left\vert H_{1}(X_{n}%
,\mathbb{Z})_{\operatorname*{tor}}\right\vert =\left\vert \mathcal{O}%
_{K}/J(\left\langle u^{n}\right\rangle )\right\vert $, where $\left\langle
u^{n}\right\rangle $ denotes the subgroup of $\mathcal{O}_{K}^{\times,+}$
which is generated by $u^{n}$; or equivalently the unique subgroup of
$\mathcal{O}_{K}^{\times,+}$ of index $n$. By Lemma \cite[Lemma 2]{otarith} we
have $J(\left\langle u^{n}\right\rangle )=(1-u^{n})$. Hence,%
\[
\left\vert \mathcal{O}_{K}/J(\left\langle u^{n}\right\rangle )\right\vert
=\left\vert N_{K/\mathbb{Q}}(1-u^{n})\right\vert =\left\vert \sigma
_{i}(1-u^{n})\right\vert \text{,}%
\]
where $\sigma_{i}$ for $i=1,2,3$ denotes the three complex embeddings (one
real, say $\sigma_{1}$, and one complex conjugate pair, say $\sigma_{2}%
,\sigma_{3}:=\overline{\sigma_{2}}$). We have%
\[
1=\left\vert N_{K/\mathbb{Q}}(u)\right\vert =\left\vert \sigma_{1}%
(u)\right\vert \left\vert \sigma_{2}(u)\right\vert ^{2}%
\]
since $u$ is a unit. If $\left\vert \sigma_{i}(u)\right\vert \leq1$ for all
$i$, then this equations forces that $\left\vert \sigma_{i}(u)\right\vert =1$
for all $i$, and then by Kronecker's Theorem $u$ must be a root of unity,
which is impossible (by Dirichlet's Unit Theorem $u$ generates the non-torsion
part of the unit group). Hence, we must have $\left\vert \sigma_{1}%
(u)\right\vert >1$ and thus $\left\vert \sigma_{2}(u)\right\vert <1 $, or the
other way round. We will now only handle the case $\left\vert \sigma
_{1}(u)\right\vert >1$ and leave the opposite case to the reader. We compute
\[
\left\vert N_{K/\mathbb{Q}}(u^{n})\right\vert =\left\vert \sigma
_{1}(u)\right\vert ^{n}\cdot\left\vert \sigma_{2}(u)\right\vert ^{2n}%
\]
and therefore%
\begin{align*}
\frac{\log\left\vert H_{1}(X_{n},\mathbb{Z})_{\operatorname*{tor}}\right\vert
}{n}  &  =\frac{\log\left\vert 1-\sigma_{1}(u^{n})\right\vert }{n}+2\frac
{\log\left\vert 1-\sigma_{2}(u^{n})\right\vert }{n}\\
&  =\frac{\log\left\vert \sigma_{1}(u^{n})\left(  \sigma_{1}(u^{n}%
)^{-1}-1\right)  \right\vert }{n}+2\frac{\log\left\vert 1-\sigma_{2}%
(u^{n})\right\vert }{n}\\
&  =\log\left\vert \sigma_{1}(u)\right\vert +\frac{\log\left\vert 1-\left(
\sigma_{1}(u)^{-1}\right)  ^{n}\right\vert }{n}+2\frac{\log\left\vert
1-\sigma_{2}(u)^{n}\right\vert }{n}%
\end{align*}
and since $\left\vert \sigma_{1}(u)^{-1}\right\vert <1$ and $\left\vert
\sigma_{2}(u)\right\vert <1$, it follows that the second and third summand
converge to zero as $n\rightarrow+\infty$. Next, since $\left\vert \sigma
_{1}(u)\right\vert >1$ and $\left\vert \sigma_{2}(u)\right\vert <1$, the
Mahler measure also satisfies $M(f)=\left\vert \sigma_{1}(u)\right\vert $,
proving Equation \ref{lcd1} in this case. Furthermore, this means that%
\[
\left\vert H_{1}(X_{n},\mathbb{Z})_{\operatorname*{tor}}\right\vert
\approx\left\vert \sigma_{1}(u)\right\vert ^{n}\qquad\text{for large }n
\]
with $\left\vert \sigma_{1}(u)\right\vert >1$, so the torsion homology of
$H_{1}$ grows strictly exponentially as an asymptotic. As explained, we leave
the other case to the reader. The argument is entirely symmetric, just
swapping the roles of $\sigma_{1}$ and $\sigma_{2}$.
\end{proof}

The previous proof explains the intense torsion growth which we had
computationally observed in \cite{otarith}, but which at that time had
appeared somewhat mysterious.

This type of argument is not new, however, it might be new in the field of
complex surfaces. It is a well-known type of behaviour in $3$-manifold
topology and knot invariants. In fact, it turns out that Inoue surfaces, by
the general fact that their fundamental group has a canonical epimorphism to
$\mathbb{Z}$,%
\[
\pi_{1}(X)\longrightarrow\mathbb{Z}%
\]
form an example of a space with an \textquotedblleft augmented
group\textquotedblright\ as fundamental group, in the sense of Silver and
Williams \cite{MR1955605}. One can rephrase the previous theorem in such a way
that it becomes a special case of \cite[Prop. 2.5]{MR1955605}. To this end,
note that $\prod_{\zeta^{n}=1}\triangle(\zeta)$ in \cite[Equation
2.2]{MR1955605}, can also be rewritten as a resultant, and the previous proof
can alternatively be spelled out as a computation of exactly this resultant.
We will not go into this in detail since the above proof is quicker than
citing \cite[Prop. 2.5]{MR1955605}. Nonetheless, this elucidates the general picture.

\bibliographystyle{amsalpha}
\bibliography{ollinewbib}

\end{document}